\documentclass[11pt,a4paper]{amsart}
\usepackage[utf8]{inputenc}
\usepackage[T2A]{fontenc}
\usepackage[russian,english]{babel}
\usepackage{amsfonts,amssymb,amscd}
\usepackage{amsthm}
\usepackage{mathrsfs}
\usepackage{graphics}
\usepackage{hyperref}

\newtheorem{proposition}{Proposition}
\newtheorem{theorem}{Theorem}
\newtheorem{lemma}{Lemma}

\newtheorem{corollary}{Corollary}

\theoremstyle{definition}

\theoremstyle{remark}
\newtheorem*{remark*}{Remark}
\newtheorem*{question*}{Question}

\newcommand \Conf {{\mathrm {Conf}}}

\DeclareMathOperator{\supp}{supp}

\DeclareMathOperator{\tr}{tr}

\newcommand{\HH}{\mathscr {H}}

\newcommand{\dH}{{\dot H_{1/2}}}
\newcommand{\Ha}{{\mathfrak{H}}}
\newcommand{\Hb}{{\mathbf{H}}}
\newcommand{\SINE}{\mathscr{S}}

\renewcommand \Prob {{\mathbb P}}

\begin{document}
\title[The expectation of a multiplicative functional]{The expectation of a multiplicative functional under the sine-process}

\author[A. I. Bufetov]{Alexander I. Bufetov}

\address{Steklov Mathematical Institute of the Russian Academy of Sciences,\newline\hspace*{\parindent}Moscow, Russia\newline
	\hspace*{\parindent}The Kharkevich Institute for Information Transmission Problems,\newline\hspace*{\parindent}Moscow, Russia\newline
	\hspace*{\parindent}Aix-Marseille Universit\'e, CNRS, Institut de Mathématiques de Marseille,\newline\hspace*{\parindent}Marseille, France}

\email{bufetov@mi-ras.ru}

\date{}

\begin{abstract}
	An explicit expression for the expected value of a regularized multiplicative functional under the sine-process is obtained by passing to the scaling limit in the Borodin---Okounkov---Geronimo---Case formula.\\[2pt]
	MSC: 60B20\\
	Keywords: sine-process, multiplicative functional, Wiener---Hopf operator, Borodin---Okounkov---Geronimo---Case formula
\end{abstract}

\maketitle

\bigskip

{\raggedright\relax\leftskip=0.35\textwidth\relax\textit{\foreignlanguage{russian}{Светлой памяти моих дорогих учителей:\\ Бориса Марковича Гуревича (1938--2023) и\\ Анатолия Моисеевича Вершика (1933--2024)}}\par}

\bigskip

{\raggedright\relax\leftskip=0.35\textwidth\relax\textit{In loving memory of my teachers,\\ Boris Markovich Gurevich (1938--2023) and\\ Anatoly Moiseyevich Vershik (1933--2024)}\par}

\section{Introduction}

\subsection{Statement of the main result}

The sine-process is the scaling limit of radial parts of Haar measures on the unitary groups of growing dimension. Expectations of multiplicative functionals with respect to the sine-process admit an explicit expression for observables of Sobolev regularity $1/2$ and having bounded Hilbert transform. The explicit expression is obtained as the scaling limit of the Borodin---Okounkov---Geronimo---Case formula for Toeplitz determinants. Recall that the Borodin---Okounkov---Geronimo---Case formula gives the remainder term in the Strong Szeg\H{o} Theorem in the form of Ibragimov.

Recall that the sine-process, denoted by $\Prob_{\SINE}$, is a determinantal point process with the sine kernel
\begin{equation*}
	\SINE(x,y)=\frac{\sin\pi(x-y)}{\pi(x-y)},
\end{equation*}
which is the kernel of the projection operator acting in $L_2(\mathbb{R})$ and whose range is the Paley---Wiener space
\begin{equation*}
	\mathscr{PW}=\{f\in L_2(\mathbb{R}): \supp \widehat f\subset [-\pi,\pi])\}. 	
\end{equation*} 
In other words, the sine-process is a measure on the space of configurations $\Conf(\mathbb{R})$, that is, the space of subsets $X\subset\mathbb{R}$ without accumulation points. The sine-process is uniquely defined by the condition
\begin{equation}\label{eq:def-sin-proc}
\mathbb{E}_{\Prob_{\SINE}}\prod_{x\in X}(1+f(x))=\det(1+f\SINE)
\end{equation}
valid for any bounded Borel function $f$ with compact support. Theorem~\ref{thm:lim-BOGC} below gives a convenient expression for the expectation~\eqref{eq:def-sin-proc} for $1/2$-Sobolev regular functions $f$ with bounded Hilbert transform.

To any Borel bounded function $f$ with compact support assign an additive functional $S_f$ on $\Conf(\mathbb{R})$ by the formula
\begin{equation*}
	S_f(X)=\sum_{x\in X}f(x);
\end{equation*}
the series in the right-hand side contains only finite number of non-zero terms. The variance of the additive functional $S_f$ is given by the formula
\begin{equation*}
	\operatorname{Var}_{\Prob_{\SINE}}S_f=\frac{1}{2}\iint_{\mathbb{R}^2}|f(x)-f(y)|^2\cdot |\Pi(x,y)|^2\,dxdy.
\end{equation*}
Following \cite{Bufetov-Quasi}, we define the space $\dH(\SINE)$ as the completion of the family of compactly supported smooth functions on $\mathbb{R}$ with respect to the norm $\|\,\cdot\,\|_{\dH(\SINE)}$ given by the formula
\begin{equation*}
	\|f\|_{\dH(\SINE)}^2=
	\iint_{\mathbb{R}^2}|f(x)-f(y)|^2\cdot |\Pi(x,y)|^2\,dxdy.
\end{equation*}
By definition, the correspondence $f\mapsto S_f-\mathbb{E}_{\Prob_{\SINE}}S_f$ is extended by continuity onto the entire space $\dH(\SINE)$. Let us denote
\begin{equation*}
	\overline{S}_f=S_f-\mathbb{E}_{\Prob_{\SINE}}S_f
\end{equation*}
and refer to $\overline{S}_f$ as to the regularized additive functional for the function $f\in\dH(\SINE)$.

Below we give an explicit formula for the exponential moments of regularized additive functionals with $1/2$-Sobolev regular functions $f$ for the sine-process. Let us recall some basic definitions.

We use the following convention for the Fourier transform on the real line:
\begin{equation*}
	\widehat{f}(s)=\int_{\mathbb{R}}e^{-i\lambda s} f(\lambda)\,d\lambda,\quad
	f(\lambda)=\frac{1}{2\pi}\int_{\mathbb{R}}e^{i\lambda s}\widehat{f}(s)\,ds
\end{equation*}
Denote by the symbol $\widetilde{\phantom{a}}$ the reflection with respect to zero:
\begin{equation*}
	\widetilde{f}(\lambda)=\frac{1}{2\pi}\int_{\mathbb{R}} \widehat{f}(-u)e^{iu\lambda}\,du=f(-\lambda).
\end{equation*}
Define the space of the Sobolev type $\dH(\mathbb{R})$ as the completion of the family of smooth compactly supported functions with respect to the norm
\begin{equation*}
	\|f\|_{\dH(\mathbb{R})}^2=\int_{\mathbb{R}}|u|\cdot |\widehat{f}(u)|^2\,du=
	\iint_{\mathbb{R}^2}\biggl|\frac{f(\xi)-f(\eta)}{\xi-\eta}\biggr|^2\,d\xi d\eta.
\end{equation*}
By the symbol $\langle{\,\cdot\,},{\,\cdot\,}\rangle_{\dH(\mathbb{R})}$ we denote the bilinear form given by the formula
\begin{equation*}
	\langle f_1,f_2\rangle_{\dH(\mathbb{R})}= \int_{\mathbb{R}}|u|\cdot \widehat{f_1}(u)\widehat{f_2}(-u)\,du.
\end{equation*}
Therefore,
\begin{equation*}
	\|f\|_{\dH(\mathbb{R})}^2=
	\langle f,\overline{f}\rangle_{\dH(\mathbb{R})}.
\end{equation*}
Further, for a function $f\in\dH(\mathbb{R})$ let $f_+$, $f_-$ be the functions defined by the formulae
\begin{equation*}
	\widehat{f_+}=\widehat f \cdot\chi_{(0,\infty)},\quad
	\widehat{f_-}=\widehat f \cdot\chi_{(-\infty,0)}.
\end{equation*}
Finally, let
\begin{equation}\label{eq:h}
	h=e^{f_- - f_+}.
\end{equation}
Clearly, $f_--f_+$ is the Hilbert transform of the function $f$ multiplied by $\sqrt{-1}$.

For a function $r\in L_2(\mathbb{R})\cap L_\infty(\mathbb{R})$ denote by $\Ha(r)$ the continual Hankel operator, acting by the formula
\begin{equation*}
	\Ha(r)\varphi(s)=\frac{1}{2\pi}\int_0^{+\infty} \widehat{r}(s+t)\varphi(t)\,dt. 
\end{equation*}

Denote by the symbol $\HH(1/2,\infty)$ the completion of the space of smooth compactly supported functions on $\mathbb{R}$ with respect to the norm
\begin{equation*}
	\|f\|_{\HH(1/2,\infty)}=\|f\|_{L_\infty(\mathbb{R})}+\|f\|_{\dH(\mathbb{R})}.
\end{equation*}
One can see from the definition that for $h\in\HH(1/2,\infty)$ the operator $\Ha(h)$ is Hilbert---Schmidt.

Further, $f_1,f_2\in\HH(1/2,\infty)$ implies $f_1f_2\in\HH(1/2,\infty)$. Therefore, $\exp(f)-1\in\HH(1/2,\infty)$ holds if $f\in\HH(1/2,\infty)$.

Now the expectation of a multiplicative functional of the sine-process, corresponding to a $1/2$-Sobolev regular function, is given as follows.

\begin{theorem}\label{thm:lim-BOGC}
	Let $f\in \dH(\mathbb{R})$ satisfy $f_--f_+\in L_\infty(\mathbb{R})$. Then
	\begin{multline}	\label{eq:thm-lim-BOGC}
		\mathbb{E}_{\Prob_{\SINE}}\exp(\overline{S}_f)=		\exp\Bigl(\frac{1}{4\pi^2}\langle f_+,\widetilde{f_-} \rangle_{\dH(\mathbb{R})}\Bigr)\cdot{}\\
		{}\cdot
		\det\biggl(1-\chi_{(1,+\infty)}\Ha\biggl(h\biggl(\frac{\cdot}{2\pi}\biggr)\biggr)\Ha\biggl(\widetilde{ h^{-1}}\biggl(\frac{\cdot}{2\pi}\biggr)\biggr)\chi_{(1,+\infty)}\biggr).
	\end{multline}
\end{theorem}

The condition $f_--f_+\in L_\infty(\mathbb{R})$ immediately implies $h=\exp(f_--f_+)\in\HH(1/2,\infty)$.

The boundedness condition may be omitted if $f$ is real-valued.
\begin{corollary}
	The formula \eqref{eq:thm-lim-BOGC} holds if $f\in\dH(\mathbb{R})$ is real-valued.
\end{corollary}

\begin{proof}[Proof of the corollary.]
	If $f$ is real-valued, then $|h|\equiv 1$, $h\in\HH(1/2,\infty)$ and $\|\Ha(h)\|\le 1$, so the formula~\eqref{eq:thm-lim-BOGC} holds even if $f$ is unbounded.
\end{proof}

Theorem~\ref{thm:lim-BOGC} is proven by passing to a scaling limit in its discrete counterpart, the Borodin---Okounkov---Geronimo---Case formula.

It is enough to prove Theorem~\ref{thm:lim-BOGC} for smooth functions with compact support. The general case immediately follows from the continuity of both sides in~\eqref{eq:thm-lim-BOGC} in the space $\HH(1/2,\infty)$.

\begin{remark*}
	Under more restrictive assumptions on $f$ an equivalent of the formula~\eqref{eq:thm-lim-BOGC} has been obtained by Basor and Chen~\cite{BasorChen}. A special case of the formula~\eqref{eq:thm-lim-BOGC} has been used in~\cite{Bufetov} to prove that almost all realizations of the sine-process have excess one in the Paley---Wiener space.
\end{remark*}

\subsection{Historical remarks}
Gabor Szeg\H{o} proved the Polya conjecture, which is now known as the First Szeg\H{o} Theorem~\cite{Szego1}, in 1915, the same year he enlisted in Royal Honv\'ed cavalry as a volunteer. The Second (or Strong) Szeg\H{o} Theorem~\cite{Szego2} was proven 37 years lated, in 1952. For a review of its history see~\cite{DeiftItsKras,Simon}. Necessary and sufficient conditions for the validity of the~Strong Szeg\H{o} Theorem  were established by Ibragimov~\cite{Ibragimov,GolIbr}. Geronimo and Case~\cite{GeronimoCase} have proved formula~\eqref{eq:BOGC-formula} in 1979. It was later rediscovered in 2000 by Borodin and Okounkov~\cite{BorodinOkounkov}. Now there exists several proofs of the Borodin---Okounkov---Geronimo---Case formula~\cite{BasorWidom,Boettcher1,Boettcher2}, but the question of necessary and sufficient conditions for the formula is still open.

\subsection{Acknowledgments}

I am deeply grateful to Sergey Gorbunov and Roman Romanov for useful discussions and to anonymous referee for the thorough review of the paper and numerous improvements suggested.

\section{Beginning of the proof of Theorem~\ref{thm:lim-BOGC}: Borodin---Okounkov---Geronimo---Case formula}

Let us recall the Borodin---Okounkov---Geronimo---Case formula.

Denote $\mathbb{T}=\mathbb{R}/2\pi\mathbb{Z}$. For a function $F\in L_1(\mathbb{T})$ with Fourier expansion
\begin{equation*}
	\widehat F(k)=\frac{1}{2\pi}\int_{\mathbb{T}}F(\theta)e^{-ik\theta}\,d\theta,
\end{equation*}
we define the Toeplitz operator $T(F)$ with symbol $F$ as the operator acting on functions $\varphi$ with finite support by the formula
\begin{equation*}
	T(F)\varphi(k)=\sum_{l\in\mathbb{N}} \widehat{F}(k-l)\varphi(l),
\end{equation*}
The Hankel operator $\Hb(F)$ with the symbol $F$ is defined by the formula
\begin{equation*}
	\Hb(F)\varphi(k)=\sum_{l\in\mathbb{N}} \widehat{F}(k+l-1)\varphi(l).
\end{equation*}
If $F\in L_\infty(\mathbb{T})$, operators $T(F)$ and $\Hb(F)$ are bounded on $l_2(\mathbb{N})$.
As in the case of the real line, by the symbol $\widetilde{\phantom{a}}$ we denote the reflection with respect to zero:
\begin{equation*}
	\widetilde{F}(\theta)=\sum_{k\in\mathbb{Z}} \widehat{F}(-k)e^{ik\theta}=F(-\theta).
\end{equation*}
Let $D_n(F)$ stand for the $n\times n$ Toeplitz determinant corresponding to a symbol $F$, that is,
\begin{equation*}
	D_n(F)=\det (T(F)_{ij})_{i,j=1,\dots,n}.
\end{equation*}
The Andreief formula~\cite{Andreief}
\begin{multline*}
	\int_X\cdots\int_X
	\det(\varphi_i(x_j))_{i,j=1,\dots,n}\cdot
	\det(\psi_i(x_j))_{i,j=1,\dots,n}\,dx_1\dots dx_n={}\\
	{}=\det\biggl(\int_X \varphi_i(x)\psi_j(x)\,dx\biggr)_{i,j=1,\dots,n}
\end{multline*}
implies that for any function $G\in L_1(\mathbb{T})$ we have
\begin{equation}\label{eq:Andreev-DnG}
	D_n(G)=\frac{1}{n!}\int_{\mathbb{T}}\cdots\int_{\mathbb{T}}
	|e^{i\theta_k}-e^{i\theta_l}|^2\cdot
	\prod_{k=1}^{n} G(\theta_k)\frac{d\theta_k}{2\pi}=\det(1+(G-1)K_n),
\end{equation}
where
\begin{equation*}
	K_n(\theta,\theta')=\frac{\sin\frac{n+1}{2}(\theta-\theta')}{\sin\frac{1}{2}(\theta-\theta')}
\end{equation*}
is the $n$-th Dirichlet kernel.

As in the previous section, introduce the Sobolev space of order $1/2$ on the unit circle $\mathbb{T}=\mathbb{R}/2\pi\mathbb{Z}$, the respective seminorm and the bilinear form as follows:
\begin{gather*}
	H_{1/2}(\mathbb{T})=\{f\in L_2(\mathbb{T}): \sum_{k\in\mathbb{Z}}|k|\cdot |\widehat{F}(k)|^2<+\infty\},\\
	\|F\|_{\dH(\mathbb{T})}^2=\sum_{k\in\mathbb{Z}}|k|\cdot |\widehat{F}(k)|^2,\\
	\langle F_1,F_2\rangle_{\dH(\mathbb{T})}= \sum_{k\in\mathbb{Z}}|k|\cdot \widehat{F_1}(k)\widehat{F_2}(-k).
\end{gather*}
Therefore,
\begin{equation*}
	\|F\|_{\dH(\mathbb{T})}^2=
	\langle F,\overline{F}\rangle_{\dH(\mathbb{T})}.
\end{equation*}
Following Borodin and Okounkov, consider a square-integrable function on the unit circle $\mathbb{T}=\mathbb{R}/2\pi\mathbb{Z}$
\begin{equation*}
	F(\theta)=\sum_{k\in\mathbb{Z}} \widehat{F}(k)e^{ik\theta}
\end{equation*}
with zero average: $\widehat F(0)=0$.
Denote
\begin{equation*}
	F_+(\theta)=\sum_{k>0} \widehat{F}(k)e^{ik\theta},\quad
	F_-(\theta)=\sum_{k<0} \widehat{F}(k)e^{ik\theta}.
\end{equation*}
Define the function $N$ by the formula
\begin{equation*}
	N(\theta)=\exp(F_-(\theta)-F_+(\theta)).
\end{equation*}
We are now ready to formulate the Borodin---Okounkov---Geronimo---Case theorem.

\begin{theorem}[\cite{BorodinOkounkov,GeronimoCase}]
Let $F\in H_{1/2}(\mathbb{T})$, $\widehat F(0)=0$, $F_--F_+\in L_\infty(\mathbb{T})$. Then we have
\begin{multline}\label{eq:BOGC-formula}
	D_n(\exp(F))={}\\
	{}=
	\exp\biggl(\sum_{k>0} k\widehat{F}(k)\widehat{F}(-k)\biggr)\cdot
	\det\bigl(1-\chi_{[n+1,+\infty)}\Hb(N)\Hb(\widetilde{N^{-1}})\chi_{[n+1,+\infty)}\bigr).
\end{multline}
\end{theorem}

If $F$ is real-valued, then $|N(\theta)|=|\widetilde{N^{-1}}(\theta)|=1$ for all $\theta\in\mathbb{T}$.
It is clear from definitions that the Hankel operators $\Hb(N)$ and $\Hb(\widetilde{N^{-1}})$ are adjoint to each other.

\begin{corollary}
	If $F\in H_{1/2}(\mathbb{T})$ is real-valued, then~\eqref{eq:BOGC-formula} holds.
\end{corollary}

Borodin and Okounkov emphasize that the identity~\eqref{eq:BOGC-formula} can be considered as an equality between formal power series in $\widehat{F}(k)$.

\begin{remark*}
	The Second Szeg\H{o} theorem for $G=\exp(F)$, stating that
	\begin{equation*}
		\lim_{n\to\infty}D_n(G)=\exp\biggr(\sum_{k\in\mathbb{Z}}k\widehat{F}(k)\widehat{F}(-k)\biggr),
	\end{equation*}
	holds under the assumptions $F\in H_{1/2}(\mathbb{T})$, $\widehat{F}(0)=0$. If the function $F$ is not assumed to be real-valued, the additional condition $F_--F_+\in L_\infty(\mathbb{T})$ is required in the proof of the Borodin---Okounkov---Geronimo---Case formula to ensure that the respective Hankel operators are Hilbert---Schmidt.
\end{remark*}

\begin{question*}
	Is the condition $F_--F_+\in L_\infty(\mathbb{T})$ necessary for the Borodin---Okounkov---Geronimo---Case formula to hold? What are the necessary and sufficient conditions?
\end{question*}

\begin{remark*}Our definition slightly differs from the convention of Borodin and Okounkov, who considered the function $F_-(\pi-\theta)-F_+(\pi-\theta)$. However, the resulting kernels differ by a gauge factor of $(-1)^{i+j}$ and, therefore, have the same Fredholm determinants.
\end{remark*}

\section{Scaling limit of the Borodin---Okounkov---Geronimo---Case formula}

Our next step is to pass to the scaling limit in the Borodin---Okounkov---Geronimo---Case formula under scaling $R_n(\theta)=r(n\theta)$.

Let $r$ be a smooth compactly supported function with zero average. For large enough $n\in\mathbb{N}$ the support of the function $r(n\varphi)$ lies in the interval $(-\pi, \pi)$. Define the functions $R_n$ on $\mathbb{T}=\mathbb{R}/2\pi\mathbb{Z}$ by the formula
\begin{equation}\label{eq:Rn}
	R_n(\theta)=r(n\theta),\quad \theta\in[-\pi,\pi]. 	
\end{equation}
Introduce decompositions $r=r^+-r^-$, $R_n=R_n^++R_n^-$ into the positive and negative harmonics:
\begin{equation*}
	\supp \widehat{r^+}^{\mathbb{R}}\subset [0,+\infty),\quad
	\supp \widehat{r^-}^{\mathbb{R}}\subset (-\infty,0].
\end{equation*}
Here and below we denote Fourier transforms on $\mathbb{R}$ and on $\mathbb{T}=\mathbb{R}/2\pi\mathbb{Z}$ by $\widehat f^{\:\mathbb{R}}$ and $\widehat f^{\:\mathbb{T}}$ respectively.
By definitions, we have 
\begin{equation*}
	\widehat{R_n}^{\mathbb{T}}(k)=\frac{1}{2\pi n} \widehat{r}^{\,\mathbb{R}}\biggl(\frac{k}{n}\biggr).
\end{equation*}

\begin{lemma}Let $r$ be a smooth compactly supported function on $\mathbb{R}$, $h=\exp(r^--r^+)$. 
	Then we have
	\begin{multline}\label{eq:lim-DnRn}
		\lim_{n\to\infty} D_n(\exp(R_n))={}\\
		{}=
		\exp\biggl(\frac{1}{4\pi^2}\langle r^+,r^-\rangle_{\dH(\mathbb{R})}\biggr)\cdot
		\det\bigl(1-\chi_{[1,+\infty)}\Ha(h) \Ha(\widetilde{h^{-1}})\chi_{[1,+\infty)}\bigr).
	\end{multline}
\end{lemma}

\begin{corollary}
	The relation~\eqref{eq:lim-DnRn} holds for any function $r\in\HH(1/2,\infty)$ with compact support.
\end{corollary}

The corollary immediately follows from the lemma since both sides of the equality \eqref{eq:lim-DnRn} are continuous in the space~$\HH(1/2,\infty)$.

Before passing to the proof of the lemma, let us make the following observation on the convergence of Fredholm determinants. 
Let $K$ be a trace-class operator acting on a separable Hilbert space. 
Denote $|K|=\sqrt{K^*K}$. 
For $l\ge 1$ let $\bigwedge^{\!l}K$ stand for the exterior power of $K$. 
The exterior powers are also of trace class and we have
\begin{equation}\label{eq:tr-exterior}
	\tr\bigwedge\nolimits^{\!l} K\le \frac{(\tr|K|)^l}{l!}.
\end{equation}
Further, let $K$ and $K_n$ be trace-class operators, which may act on different Hilbert spaces. In order to establish the convergence
\begin{equation*}
	\lim_{n\to\infty}\det(1+K_n)=\det(1+K)
\end{equation*}
it is enough to verify the convergence
\begin{equation*}
	\lim_{n\to\infty}\tr\bigwedge\nolimits^{\!l} K_n=\tr\bigwedge\nolimits^{\!l} K,
\end{equation*}
and the uniform estimate
\begin{equation}\label{eq:est-sum-traces}
	\sum_{l=1}^\infty \sup_{n\in\mathbb{N}}\bigl|\tr\bigwedge\nolimits^{\!l} K_n\bigr|<+\infty.
\end{equation}
Moreover, if $K_n$ are positive self-adjoint operators, one can see from the formula~\eqref{eq:tr-exterior} that the condition
\begin{equation*}
	\sup_{n\in\mathbb{N}} \tr K_n<+\infty
\end{equation*}
is sufficient for \eqref{eq:est-sum-traces} to hold.

Consider the discrete Hankel operator with a smooth symbol $r$. Below we will use the simple estimate
\begin{equation*}
	\bigl|\tr \bigwedge\nolimits^{\!l} \Hb(r)\bigr|\le\frac{(\|\widehat{r}^{\mathbb{T}}\|_{L_1})^l}{l!}.
\end{equation*}
For the continual Hankel operator with a smooth symbol $r$ we will employ the similar estimate
\begin{equation*}
	\bigl|\tr\bigwedge\nolimits^{\!l} \Ha(r)\bigr|\le\frac{((2\pi)^{-1}\cdot\|\widehat{r}^{\mathbb{R}}\|_{L_1})^l}{l!}.
\end{equation*}

\begin{proof}[Proof of the lemma.]
	Recall that the functions $R_n$ are given by the formula~\eqref{eq:Rn}.
	For their Sobolev seminorms we have the equality
	\begin{equation*}
		\langle R_n^+,R_n^-\rangle_{\dH(\mathbb{T})}=
		\sum_{k>0} k \widehat{R_n}^{\mathbb{T}}(k)\widehat{R_n}^{\mathbb{T}}(-k)=
		\frac{1}{4\pi^2}\sum_{k>0}\frac{k}{n^2}\widehat{r}^{\,\mathbb{R}}\biggl(\frac{k}{n}\biggr)\widehat{r}^{\,\mathbb{R}}\biggl(-\frac{k}{n}\biggr),
	\end{equation*}
	whence the limit is
	\begin{equation*}
		\lim_{n\to\infty}\langle R_n^+,R_n^-\rangle_{\dH(\mathbb{T})}=
		\frac{1}{4\pi^2}\langle r^+,r^-\rangle_{\dH(\mathbb{R})}.
	\end{equation*}
	Similarly, for exterior powers of the operators $\Hb(R_n)$ we have
	\begin{multline*}
		\tr\bigwedge\nolimits^{\!l} \Hb(R_n)={}\\
		{}=		\frac{1}{l!}\sum_{i_1,\dots,i_l\in\mathbb{N}}
		\widehat{R_n}^{\mathbb{T}}(i_1+i_2-1)\cdots 
		\widehat{R_n}^{\mathbb{T}}(i_{l-1}+i_l-1)\widehat{R_n}^{\mathbb{T}}(i_{l}+i_1-1)={}\\
		{}=
		\frac{1}{l!(2\pi n)^l}\sum_{i_1,\dots,i_l\in\mathbb{N}}
		\widehat{r}^{\,\mathbb{R}}\biggl(\frac{i_1+i_2-1}{n}\biggr)\cdots
		\widehat{r}^{\,\mathbb{R}}\biggl(\frac{i_{l-1}+i_l-1}{n}\biggr)
		\widehat{r}^{\,\mathbb{R}}\biggl(\frac{i_l+i_1-1}{n}\biggr).
	\end{multline*}
	The smoothness of the function $r$ immediately yields that  $\sup_{n\in\mathbb{N}}\|\widehat{R_n}\|_{L_1}<+\infty$.
	From the identity
	\begin{multline*}
		\tr\bigwedge\nolimits^{\!l} \Ha(r)={}\\
		{}=
		\frac{1}{l!(2\pi)^l}
		\int_0^{+\infty}\!\!\cdots\int_0^{+\infty}
		\widehat{r}^{\,\mathbb{R}}(s_1+s_2)\cdots
		\widehat{r}^{\,\mathbb{R}}(s_{l-1}+s_l)
		\widehat{r}^{\,\mathbb{R}}(s_l+s_1)\,ds_1\dots ds_l,
	\end{multline*}
	we see that
	\begin{equation*}
		\lim_{n\to\infty} \tr\bigwedge\nolimits^{\!l} \Hb(R_n)=\tr\bigwedge\nolimits^{\!l} \Ha(r),
	\end{equation*}
	which implies that
	\begin{align*}
		\lim_{n\to\infty} \det(1+\Hb(R_n))&{}=\det(1+\Ha(r)),\\
		\lim_{n\to\infty} \det(1-\Hb(R_n))&{}=\det(1-\Ha(r)).	
	\end{align*}
	Assume now that we have two smooth functions $r^{(1)}$, $r^{(2)}$ with support lying in the interval $(-\pi,\pi)$. As above, we denote $R_n^{(1)}(\theta)=r^{(1)}(n\theta)$, $R_n^{(2)}(\theta)=r^{(2)}(n\theta)$. The proof of the limit relation
	\begin{equation*}
		\lim_{n\to\infty} \det\bigl(1+\Hb(R^{(1)}_n)\Hb(R^{(2)}_n)\bigr)=\det\bigl(1+\Ha(r^{(1)})\Ha(r^{(2)})\bigr).
	\end{equation*}
	and the similar formula for the restricted Hankel operators
	\begin{multline*}
		\lim_{n\to\infty} \det\bigl(1-\chi_{[n+1,+\infty)}\Hb(R^{(1)}_n)\Hb(R^{(2)}_n)\chi_{[n+1,+\infty)}\bigr)={}\\
		{}=
		\det\bigl(1-\chi_{[1,+\infty)}\Ha(r^{(1)})\Ha(r^{(2)})\chi_{[1,+\infty)}\bigr).
	\end{multline*}
	is parallel to the considerations above. The Borodin---Okounkov---Gero\-nimo---Case formula~\eqref{eq:BOGC-formula} concludes the proof of the lemma.
\end{proof}

\section{Scaling of Fredholm determinants}

It remains to express the Fredholm determinant in the left-hand side of the formula~\eqref{eq:thm-lim-BOGC} as the scaling limit of Toeplitz determinants.

Let $f$ be a smooth function with compact support on $\mathbb{R}$. As above, we let $r(\,\cdot\,)=f(\,\cdot\,/2\pi)$ and $R_n(\theta)=r(n\theta)$.

\begin{proposition}\label{prop:lim-DnRn}
	We have
	\begin{equation*}
		\lim_{n\to\infty} D_n(\exp(R_n))=\det\bigl(1+(e^f-1)\SINE\bigr).
	\end{equation*}
\end{proposition}

Let $E$ be a complete separable metric space. Let $\mu_n$ be a sequence of $\sigma$-finite Radon measures on $E$ and $K_n$ be a sequence of continuous kernels on $E$. We assume that the kernels induce positive contractions on $L_2(E, \mu_n)$, which by a slight abuse of notation we also denote by $K_n$. Further, let $\mu$ be a $\sigma$-finite Radon measure on $E$ and $K$ be a continuous kernel, again inducing a positive contraction on $L_2(E, \mu_n)$.
We say that the sequence $K_n$ $F$-converges to the limit kernel $K$ if the following conditions hold:
	\begin{enumerate}
		\item $K_n\to K$ uniformly on compact subsets of $E\times E$;
		\item measures $K_n(x,x)\,d\mu_n(x)$ converge to $K(x,x)\,d\mu(x)$ in total variation. 
	\end{enumerate}
	The definition immediately implies
	\begin{proposition}
		Assume that a sequence of continuous kernels $K_n$ inducing nonnegative contractions $F$-converges to the limit kernel $K$. Then for any compact set $C\subset E$ we have that
		\begin{equation*}
			\lim_{n\to\infty}\det(1-\chi_CK_n\chi_C)=\det(1-\chi_CK\chi_C).
		\end{equation*}
	\end{proposition}
	\begin{proof}
		1. It is sufficient to verify the following equality for every $l$:
		\begin{multline*}
			\lim_{n\to\infty} \idotsint_{C^n}\det K_n(x_i,x_j)_{i,j=1,\dots,l}\prod_{i=1}^l d\mu_n(x_i)={}\\
			{}=	\idotsint_{C^n}\det K(x_i,x_j)_{i,j=1,\dots,l}\prod_{i=1}^l d\mu(x_i)
		\end{multline*}
		and then use the uniform convergence of the series in the definition of Fredholm determinants.
		
		2. Further, let $\varepsilon>0$ and consider a compact subset $C_\varepsilon=C\cap\{x\in E:K(x,x)\ge\varepsilon\}$. For $x_1,\dots,x_n\in C_\varepsilon$ we have the uniform convergence
		\begin{equation}\label{eq:det-uniform}
			\frac{\det K_n(x_i,x_j)_{i,j=1,\dots,l}}{\prod_{i=1}^l K_n(x_i,x_i)}\rightrightarrows
			\frac{\det K(x_i,x_j)_{i,j=1,\dots,l}}{\prod_{i=1}^l K(x_i,x_i)}.
		\end{equation} 
		Note that the fractions in~\eqref{eq:det-uniform} are bounded by~$1$ from above since $K_n$ are positive kernels. Using this fact, the convergence in total variation of the measures $K_n(x,x)\,d\mu_n(x)$ to the measure $K(x,x)\,d\mu(x)$ and passing to the limit as $\varepsilon\to 0$, we complete the proof.
	\end{proof}

Consider an interval $[a,b]\subset\mathbb{R}$. Proposition~\ref{prop:lim-DnRn} immediately gives
\begin{corollary}\label{cor:lim-det}
	Let $K$ be a trace-class operator on $L_2([a, b])$ with a continuous kernel $K(x, y)$. Let also $K_n$ be a sequence of operators on $\mathbb{Z}$ with a standard counting measure. Assume that for any $x,y\in[a, b]$ the following limit relation holds
	\begin{equation}\label{eq:lim-kernels}
		\lim_{n\to\infty}nK_n([nx],[ny])=K(x,y),
	\end{equation}
	where the convergence is uniform on $[a, b]$. Then we have
	\begin{equation}\label{eq:lim-det}
		\lim_{n\to\infty}\det(1+\chi_{[a,b]}K_n\chi_{[a,b]})=\det(1+K).
	\end{equation}
\end{corollary}

Now Proposition~\ref{prop:lim-DnRn} follows from the formula~\eqref{eq:Andreev-DnG} and the formula~\eqref{eq:lim-det} in Corollary~\ref{cor:lim-det}.


\begin{thebibliography}{99}
	\bibitem{Andreief}C. Andr\'eief. Note sur une relation les int\'egrales d\'efinies des produits des fonctions.
	\textit{M\'em. de la Soc. Sci. Bordeaux}, \textbf{2} (1883), 1--14.
	
	\bibitem{BasorChen}E. Basor, Y. Chen. A Note on Wiener-Hopf Determinants and the Borodin-Okounkov Identity.  	arXiv:math/0202062, 9~pp.
	
	\bibitem{BasorWidom}Basor, E.L., Widom, H. On a Toeplitz determinant identity of Borodin and Okounkov. \textit{Integr. Equ. Oper. Theory} \textbf{37} (2000), 397--401. 
	
	\bibitem{BorodinOkounkov}Borodin, A., Okounkov, A. A Fredholm determinant formula for Toeplitz determinants. \textit{Integr. Equ. Oper. Theory} \textbf{37} (2000), 386--396. 
	
	\bibitem{Boettcher1}A. Boettcher, On the determinant formulas by Borodin, Okounkov, Baik, Deift, and Rains. arXiv:math/0101008, 10~pp.
	
	\bibitem{Boettcher2}A. Boettcher, One more proof of the Borodin-Okounkov formula for Toeplitz determinants, Integral Equations Operator Theory \textbf{41} (2001), 123--125.

	\bibitem{Bufetov-Quasi}Alexander I. Bufetov,
	Quasi-symmetries of determinantal point processes.
	\textit{Annals of Probability}, \textbf{46}:2 (2018), 956--1003.
	
	\bibitem{Bufetov}Alexander I. Bufetov, The sine-process has excess one, arXiv:1912.13454, 57 pp.
	
	\bibitem{DeiftItsKras}Deift, P., Its, A. and Krasovsky, I., Toeplitz Matrices and Toeplitz Determinants under the Impetus of the Ising Model: Some History and Some Recent Results. \textit{Comm. Pure Appl. Math.}, \textbf{66}:9 (2013), 1360--1438.
	
	\bibitem{GeronimoCase}J.F. Geronimo and K.M. Case, Scattering theory and polynomials orthogonal on the unit circle. \textit{J. Math. Phys.} \textbf{20} (1979), 299--310.
	
	\bibitem{GolIbr}Б. Л. Голинский, И. А. Ибрагимов. О предельной теореме Г. Сегё, \textit{Изв. АН СССР. Сер. матем.}, \textbf{35}:2 (1971), 408--427.
	
	\bibitem{GrenSzego}Гренандер У., Сеге Г. Теплицевы формы и их приложения.
	Пер. с англ., М.: Иностранная литература, 1961. 307 с.
	
	\bibitem{Ibragimov}И. А. Ибрагимов. Об одной теореме Г. Сегё, \textit{Матем. заметки}, \textbf{3}:6 (1968),  693--702.
	
	\bibitem{Simon}B. Simon, Orthogonal Polynomials on the Unit Circle. Amer. Math. Soc., 2005, 1044 pp.
	
	\bibitem{Szego1}Szeg\"o, G. Ein Grenzwertsatz \"uber die Toeplitzschen Determinanten einer reellen positiven Funktion. \textit{Math. Ann.} \textbf{76} (1915), 490--503.
	
	\bibitem{Szego2}G. Szeg\H o, On certain hermitian forms associated with the Fourier series of a positive function,
	Festschrift Marcel Riesz, Lund 1952, 228--238.
\end{thebibliography}
\end{document}